\documentclass[11pt,leqno]{amsart}
\usepackage{amsfonts, amsmath, amsthm, amssymb,xspace}
\usepackage[breaklinks=true]{hyperref}
\usepackage{amsmath,amscd}
\theoremstyle{plain}
\newtheorem{theorem}{Theorem}[section]
\newtheorem{lemma}{Lemma}[section]
\newtheorem{prop}{Proposition}[section]
\newtheorem{cor}{Corollary}[section]

\theoremstyle{definition}

\newtheorem*{acknowledgement}{Acknowledgements}

\newcommand{\comment}[1]

\begin{document}

\title{Centralizers of Toeplitz operators with polynomial symbols}
\author{ Akaki Tikaradze}
\address{The University of Toledo, Department of Mathematics, Toledo, Ohio, USA}
\email{\tt atikara@utnet.utoledo.edu}

\begin{abstract}
In this note we describe centralizers of Toeplitz operators with polynomial symbols on the Bergman space.
As a consequence it is shown that if an element of the norm closed algebra generated by all Toeplitz operators commutes with a Toeplitz
operator of a nonconstant polynomial, then this element is a Toeplitz operator of a bounded 
holomorphic function.

\end{abstract}
\maketitle

Following usual notation, $L^2_a(D)$ will denote the Bergman space of the square integrable holomorphic
functions on the open unit disk $D=(z\in \mathbb{C}, |z|< 1)$, and $H^{\infty}(D)$ will denote
the set of bounded holomorphic functions on $D$. Recall that $L^2_a(D)$ is a Hilbert space under the inner product
$\langle f, g\rangle=\int_{D}f\bar{g}dA$ with respect to the standard Lebesgue measure 
with measure of the unit disc being 1. 
Elements $\sqrt{n+1}z^n, n\geq 0$ form an orthonormal basis of $L^2_a(D).$ Recall also
that for any bounded function $f\in L^{\infty}(D)$, one can define
a bounded linear operator, called the Toeplitz operator $T_{f}:L^2_a\to L^2_a$ with symbol $f$ defined as follows: $T_f(g)=P(fg),g\in L^2_a(D),$ where $P$ is the orthogonal projection from $L^2(D)$ to $L^2_a(D).$
We will consider a $C^{*}$-algebra generated by all Toeplitz operators, which becomes
a $C^{*}$-subalgebra of the algebra of all bounded operators on $L^2_a(D).$ We will refer
to this algebra as the Toeplitz algebra.

 Cuckovic \cite{C} and Cuckovic-Fan \cite{CF} proved that centralizers of $T_{h}$ in the Toeplitz algebra are Toeplitz operators with bounded analytic symbols, provided that
 $h=z^m$ for some $m>0$ \cite{C}, or $h=z+\sum_{i=2}^na_iz^i, a_i\geq 0$ for all $i.$ 
Motivated by these results, we show the following

\begin{theorem}\label{Main} Let $S:L^2_a(D)\to L^2_a(D)$ be a bounded linear operator 
which commutes
with $T_{h(z^m)},$ where $h(z)$ is a polynomial which is not of the form $h_1(z^l),$
 with $h_1$ a polynomial and $l$ a positive integer $l>1$. 
Then, $[S, T_{z^m}]=0.$ In particular, if in addition $S$ is compact, then $S=0.$
\end{theorem}

\begin{cor} \label{appl} Let $f(z)$ be an arbitrary nonconstant polynomial, if $S$ is an element of the Toeplitz
algebra which commutes with $T_f,$ then $A=T_g,$ for some $g\in H^{\infty}(D).$

\end{cor}

The proof will follow closely ideas of Cuckovic \cite{C}, and Cuckovic-Fan \cite{CF}.
The key step is the following.

\begin{prop}\label{circle} Let $h(z)\in \mathbb{C}[z]$ be a polynomial which may not be written
as a polynomial in $z^m$ for any $m>1.$ Then there is a nonempty open set $U\subset D,$
such that  $\frac{h(z)-h(w)}{z-w}\neq 0,$ for all $w\in U, z\in \bar{D}.$
\end{prop}

\begin{proof}
The open mapping property of $h(z)$ implies that the boundary of $h(\overline{D})$ is a subset of $h(S^1)$ (where $S^1$ is the unit circle, the boundary of $D$) and is disjoint from $h(D).$ Since a theorem of Quine \cite{Q}
states that there are only finitely many pairs $(z, w),$ such that $z\neq w, z, w\in S^1,
h(z)=h(w),$ we may conclude that there is a point $w\in S^1$, such that $h(w)$ belongs to the boundary of $h(\overline{D}), \frac{\partial }{\partial z}h(w)\neq 0$ and for all $z\neq w, z\in S^1, f(z)\neq f(w).$
But, this implies that $h(z)\neq f(w),$ for all $z\in \overline{D}.$ This implies that
there is a nonempty open set $U\in D$ with the desired property. 

\end{proof}

\begin{prop}\label{cuck} Suppose that $h\in \mathbb{C}[z]$ satisfies the conclusion of Proposition \ref{circle},
 then any bounded operator 
$S:L^2_a(D)\to L^2_a(D),$ which commutes with $T_h$ must be of the form $T_f,$ for some
$f\in H^{\infty}(D).$
\end{prop}
The following proof is contained in \cite{CF}.
\begin{proof}
Recall that for any $g(z)\in L^2_a(D), \langle g, K_z\rangle=g(z)$, where $K_z$ is the reproducing kernel.
This gives $T_h^{*}K_w=\overline{f(w)}K_w,$ \\in particular, $T^{*}_{h(z)-h(w)}K_w=0.$
Thus, since $S^{*}$ and $T^{*}_h$ commute, we have $T^{*}_{f(z)-f(w)}(S^{*}K_w)=0.$
Which means that $S^{*}K_{w}$ is orthogonal to the image of $T_{h(z)-h(w)},$ 
which by our proposition is $(z-w)L^2_a(D),$ for all $w\in U.$ But, since $K_w$
is also orthogonal to the above, we have that $S^{*}K_w=\psi(w)K_w$, for all $w\in U,$
where $\psi(w)$ is some function on $U.$
Thus, $$\langle g, S^{*}K_w\rangle=\langle S(g), K_w\rangle=S(g)(w)=\overline{\psi(w)}g(w).$$ This implies that
$[S, T_z](g)|U=0,$ so $S$ commutes with $T_z,$ therefore $S=T_{\eta},$ for some bounded analytic
$\eta$.

\end{proof}

\begin{proof}[Proof of Theorem \ref{Main}] For any $0\leq i<m, $ consider bounded linear
operators $e_{i}, f_{i}:L^2_a(D)\to L^2_a(D)$ defined as follows:
$e_{i}(z^n)=z^{i+nm}, f_{i}(z^{i+nm})=z^n$ for all $n$. Thus, $f_{i}$ is the
composition of the orthogonal projection
of $L^2_a(D)$ on $e_i(L^2_a(D))$ with $e_i^{-1}.$ Let us put $T_{i, j}=f_jSe_i.$
It is clear
that $S_{i, j}$ commutes with $T_{h_1(z)}.$ So $S_{i, j}$  
is given
by $T_{\psi_{i, j}},$ for some bounded analytic $\psi_{i, j},$ by propositions \ref{circle}, \ref{cuck}, this
implies that $S$ commutes with $T_{z^m}.$ If in addition, $S$ is compact, then so are
operators $e_iSf_j=T_{\psi_{ij}}:L^2_a(D)\to L^2_a(D),$ which forces $\psi_{ij}=0,$ so $e_iSf_j=0$
for all $i, j,$ so $S=0.$

\end{proof}
Now we turn to the proof of Corollary \ref{appl}.

\begin{lemma} \cite{C} If $S$ belongs to the Toeplitz algebra, then $[T_z, S]$
is a compact operator.
\end{lemma}

We recall the proof for the convenience of the reader.
\begin{proof}
Since compact operators form a two sided ideal in the algebra of bounded operators,
it is enough to check that $[T_f, T_z]$ is compact for any $f\in L^{\infty}(D).$
However $[T_f, T_z]=H^{*}_{\bar{z}}H_f,$ where 
$H_f(g)=fg-T_f(g), H_f:L^2_a(D)\to {L^2}(D)^{\perp}$ is the Hankel operator with symbol $f.$
It is well-known that the Hankel operator $H_{\bar{z}}$ is compact, so we are done.

\end{proof}

Now we can easily proof Corollary \ref{appl}. If $S$ belongs to the Toeplitz algebra, and it
commutes with $T_{h}$ for a nonconstant $h\in \mathbb{C}[z]$, then by the above
$[T_z, S]$ is compact. But, since $[T_h, [T_z, S]]=0,$ Theorem \ref{Main}  implies that  
$T_z$ commutes with
$S$, forcing $S$ to be of the form $T_{\psi}$, for some bounded holomorphic $\psi.$

It is natural to ask if Theorem \ref{Main} and Corollary \ref{appl} hold for arbitrary
nonconstant bounded holomorphic functions. A positive indication in this direction is provided
by a well-known theorem of Axler-Cuckovic-Rao \cite{ACR}, which says that  if $T_{g}$ 
 commutes with $T_f$ for bounded $g$ and nonconstant holomorphic $f$ 
 (in fact for an arbitrary bounded domain, not just the unit disk $D$) , then $g$ must be holomorphic.
However, both Theorem \ref{Main} and Corollary \ref{appl} fail for arbitrary nonconstant bounded
holomorphic functions (contrary to what is claimed in \cite{L}). We present two examples below.

 The first example, due to Trieu Le, shows that Theorem \ref{Main} fails for 
 arbitrary holomorphic functions. 
 Indeed, let $g:D\to D$ be a Mobius automorphism of the unit disc
 $g(z)=\frac{z-a}{1-\bar{a}z}, a\in D, a\neq 0$, and
 let $T:L^2_a\to L^2_a$ be a bounded operator. Let us denote by $gT:L^2_a(D)\to L^2_a(D)$ an operator
 defined as follows $gT(f(z))(w)=T(f(g^{-1}(z))(g(w)), f\in L^2_a(D), w\in D.$
  Let us take an operator $T,$ which commutes with $z^n,$ such that $T$ is not
   a Toeplitz operator with an analytic symbol. Then, $gT$ commutes
 with $T_{g^n}$ and is not a Toeplitz operator with an analytic symbol.
  However, $g^n$ satisfies the condition of the proposition, namely, 
  it cannot be written as a holomorphic
 function of $z^m,$ for any $m>1.$ 

The next example shows that even Corollary \ref{appl} is false for arbitrary nonconstant holomorphic functions. Indeed,
  an example of Cowen \cite{Co} provides a bounded holomorphic function whose Toeplitz symbol
   commutes with a compact
  operator on the Hardy space. But exactly the same example works for the Bergman space setting.
   Let us recall Cowen's example for the convinience of the reader. Let $\sigma (z)=(i-1)(1+z)^{-\frac{1}{2}}.$
 Then $J(z)=\sigma^{-1}(\sigma(z)+2\pi i)$ maps $D$ to itself, continuously extends to the boundary of
 $D$ and $J(\overline{D})\subset D\cup {-1}, J(-1)=-1.$ Let $f(z)=\exp (\sigma(z))-\exp (\sigma(0)),$
then $f$ is a bounded analytic function on $D$ and $f(J(z))=f(z)$ for all $z\in D.$ Denote by
$C_{J}$ the composition operator of $J,$ so $C_{J}(f)=f(J) f\in L^2_a(D).$ Now claim is that the operator
$L=C_JT_{z+1}$ is compact. Indeed, let $||z^ng_n||=1.$ For any $\epsilon>0,$ let 
$K_{\epsilon}\subset D$ denote the set of all $z,$ such that $|1+J(z)|\geq\epsilon.$ 
Clearly $\overline{J(K_{\epsilon})}$ is compact, so there is $0<\delta<1,$ such that 
 $J(z)<\delta<1$ for all $z\in K_{\epsilon}.$
It is also clear
that $||g_n||\leq \sqrt{n+1}.$ We have 
$$\int_{K_{\epsilon}}|j^{n}(z)g_n(J(z))(J(z)+1)|^2d\mu\leq 2 \delta^{2n}||C_J||\sqrt{n+1},$$ which clearly
goes to 0 as $n\to \infty.$ On the other hand, 
$$\int_{D\setminus K_{\epsilon}}|j^{n}(z)g_n(J(z))(J(z)+1)|^2d\mu\leq \epsilon ||C_J||.$$ 
Therefore, we can conclude that $||L_{z^nL^2_a}||\to 0,$ so $L$ is compact and it commutes with
$T_{f}.$ But as it is well-known, all compact operators belong to the Toeplitz algebra, and since a nonzero
compact operator can not be a toeplitz operator with an analytic symbol, we see that Corollary \ref{appl} is false
for arbitrary analytic functions.

 Finally, we present a partial result for centralizers of $T_f$, where $f$
 is a nonconstant bounded holomorphic function. To state our result,
 we must recall that any function $g\in L^2(D)$ admits a polar
 decomposition $$g(re^{it})=\sum _{k=-\infty}^{+\infty}e^{ikt}g_k(r),$$
 where $f_r$ are radial functions.
 
 We have the following
 
\begin{prop} Suppose that an operator $S$ belongs to the algebra generated by Toeplitz
operators of the form $T_{g}, g_k=0$ for $k<<0.$ If $S$ commutes with $T_{f},$ 
where $f$ is bounded holomorphic function such that $f'(0)\neq 0$, then $S=T_\psi$ for some
$\psi\in H^{\infty}(D).$

\end{prop} 

\begin{proof}
Recall that by a computation from \cite{CL}, $T_{e^{ikt}g(r)}(z^n)$ is a multiple of $z^{n+k},$ 
 for all $k, n.$ In particular, for any $k\geq 0$ $\langle S(z^n), z^k\rangle=0$ as long as
 $n>>0.$ Let $f=a_0+a_1z+z^2f_1, a_1\neq 0, m>0, f_1\in H^{\infty}(D).$ It suffices to show that
 $S$ commutes with $z^m.$ Notice that for any $z^k,$ there exists a polynomial $\phi_k(z),$
such that $z-\phi_k(f(z))\in z^kH^{\infty}(D).$ This implies that for any $m, l\geq 0$
$$\langle [S, T_z](z^l), z^n\rangle=\langle [S, T_{z}-T_{\phi_k(f(z))}](z^l), z^n\rangle=0$$
for $k>>0.$ So, $T_{z}$ commutes with $S,$ and we are done.

\end{proof}

 \acknowledgement{I am enormously grateful to Trieu Le for many interesting discussions, in particular for
 telling me about the commuting problem for the Toeplitz algebra. I would like to thank
 Zeljko Cuckovic  for showing me his paper \cite{CF}, ideas from which played the crucial role.
 Special thanks are due to Mitya Boyarchenko for his key insight with the proof of proposition 0.1. 
  I also would like to thank
 Abdel Yousef for telling me about operators with quasi-homogeneous symbols. }

\end{document}